\documentclass[12pt]{amsart}

\usepackage[margin = 1.4in]{geometry}
\usepackage[numbered]{bookmark}
\usepackage{graphicx}
\usepackage{amsmath}
\usepackage{amsfonts}
\usepackage{amsthm} 
\usepackage{amssymb}
\usepackage{cite}
\usepackage{mathrsfs}
\usepackage{changepage}  
\usepackage[T1]{fontenc}
\usepackage{verbatim} 
\usepackage{tikz}
\usetikzlibrary{plotmarks}

\newtheorem{thm}{Theorem}[section]

\newtheorem{lem}[thm]{Lemma}
\newtheorem{prop}[thm]{Proposition}
\newtheorem{cor}[thm]{Corollary}

\newcommand{\R}{\mathbb{R}}
\newcommand{\Z}{\mathbb{Z}} 
\newcommand{\N}{\mathbb{N}}
\newcommand{\torus}{\mathbb{S}^1}
\newcommand{\dnm}{I_{n - 1}}
\newcommand{\dn}{I_n}
\newcommand{\dnp}{I_{n + 1}}
\newcommand{\dnmi}{I_{n - 1}^i}

\newcommand{\pn}{\mathcal{P}_n}
\newcommand{\pnp}{\mathcal{P}_{n+1}}

\newcommand{\hone}{\hspace{1cm}}
\newcommand{\swiatek}{\'Swi\k{a}tek}

\begin{document}

\date{}

\title[Hausdorff dimension for multicritical circle maps]{On the Hausdorff dimension of invariant measures for multicritical circle maps}

\author{Frank Trujillo}
\address{IMJ-PRG, UP7D\\ 58-56 Avenue de France\\ 75205 Paris Cedex 13\\ France}
\email{frank.trujillo@imj-prg.fr}

\maketitle

\begin{abstract}
We give explicit bounds for the Hausdorff dimension of the unique invariant measure of  $C^3$ multicritical circle maps without periodic points. These bounds depend only on the arithmetic properties of the rotation number. 
\end{abstract}

\section{Introduction}

The object of study in this work are \textit{$C^r$ multicritical circle maps} of the circle  $\torus = \R \setminus \Z$. A map $f: \torus \rightarrow \torus$ is called a  $C^r$ multicritical circle map if it is an orientation preserving circle homeomorphism of class $C^r$ having finitely many critical points $c_0, \dots, c_{n - 1}$ and taking the form
\[x \mapsto  f(c_i) +  x|x|^{p_i - 1}\]
in a suitable coordinate system around each critical point for some $p_i > 1$. The real number $p_i$ is called the \textit{power-law exponent} or \textit{criticality} of the critical point $c_i$.  Let $f : \torus \rightarrow \torus$ be an orientation-preserving circle homeomorphism and let $F : \R \rightarrow \R$ be a lift of $f$, that is, a continuous homeomorphism of $\R$ such that $F(x+1) = F(x) + 1$ and $F (x) (\bmod$ $1) = f(x)$ for all $x \in \R$. By a classical result of Poincar\'e the limit 
\[ \rho(f) = \lim_{n \to \infty} \dfrac{F^n(x)}{n} \mod 1\]
is well defined and  independent of the value $x \in \R$ initially chosen. This limit is called the \textit{rotation number} of $f$.  We say that a multicritical circle map is \textit{irrational} if its rotation number is irrational. Recall that the rotation number of a circle homeomorphism is irrational if and only the map does not possess any periodic points. Furthermore, maps with irrational rotation number admit a unique invariant measure. For a simple proof of this fact we refer the reader to \cite{furstenberg_strict_1961}. \\

It was proven by Khanin in \cite{khanin_universal_1991} that the unique invariant measure $\mu$ of  any irrational multicritical circle map $f$ is \textit{singular} with respect to the Lebesgue measure, that is, there exists a measurable set $X \subset \torus$ of zero Lebesgue measure for which $\mu(X) = 1$. An alternative proof of this fact was given  in \cite{graczyk_singular_1993} by Graczyk and \swiatek. Since this unique invariant measure is singular, it is natural to investigate its \textit{Hausdorff dimension}.    The Hausdorff dimension of any probability Borel measure on $\torus$ is defined as
\[\dim_H(\mu) := \inf \{ \dim_H(X) \mid \mu(X) = 1 \}, \] 
where $dim_H(X)$ denotes the Hausdorff dimension of the set $X \subset \torus$. We provide the formal definition of Hausdorff dimension in the next section.   \\

By a classical result of Yoccoz \cite{yoccoz_il_1984} any $C^3$ irrational critical circle map  $f$ is topologically conjugated to the irrational rotation $x \mapsto x + \rho(f)$. Hence any two irrational multicritical circle maps  with the same rotation number are conjugated to each other by a $C^0$ homeomorphism.  It was proven in \cite{estevez_real_2018} by Estevez and de Faria that if these two maps have the same number of critical points any conjugacy between them is also quasisymmetric. In the same work the authors conjecture that for $C^3$ irrational multicritical circle maps their differentiable conjugacy class depends only on their \textit{signature}. Given a multicritical circle map $f$ with $n_f$ critical points $c_0, \dots, c_{n_f - 1}$ and unique invariant measure $\mu_f$ its signature is the $(2n_f + 2)$ tuple 
\[ (n_f, \rho(f); p_0, \dots, p_{n_f - 1}; \lambda_0, \dots, \lambda_{n_f - 1}),\]
where $p_i$ is the criticality of $c_i$ and $\lambda_i = \mu_f[c_i, c_{i + 1}]$. Since the Hausdorff dimension of a set is preserved by diffeomorphisms,  the conjecture implies that the Hausdorff dimension of the unique invariant measure of $C^3$ irrational multicritical circle maps depends only on their signature. The previous conjecture is a generalization of the corresponding one for unicritical circle maps, where the only invariants are the rotation number and the criticality of the unique critical point. Guarino, Martens and De Melo \cite{guarino_rigidity_2018}  recently proved a weak version of this conjecture for $C^4$ irrational unicritical maps with the same odd criticality. \\

The main result of this paper gives a relation between the Hausdorff dimension for the unique invariant measure of a $C^3$ irrational multicritical circle map and the arithmetic properties of its rotation number.   We will recover, as a by-product of our constructions (Corollary \ref{prop: singular_measure}), a proof of the singularity of the unique invariant measure of irrational multicritical maps, in the same spirit of that in \cite{khanin_universal_1991} for the unicritical case, for maps whose rotation number is not of bounded type. Before stating our main result let us introduce some notations that will be used throughout this work.  A real number $\alpha$ is said to be \textit{Diophantine} if  there exist $\gamma > 0$ and  $\tau \geq 0$ such that 
\begin{equation}
\label{Diophantine}
\left|  \alpha -\dfrac{p}{q} \right| \geq \dfrac{\gamma}{q^{\tau+2}} \hone \text{ for all } p, q \in \Z, \, q \neq 0.
\end{equation}
Given $\gamma > 0$ and  $\tau \geq 0$ we denote by $\mathcal{D}(\gamma,\tau)$ the set of real numbers $\alpha$ verifying (\ref{Diophantine}).  A number in $\mathcal{D}(\gamma,\tau)$ is called a \textit{Diophantine number of type} $(\gamma, \tau)$. Let 
\[ \mathcal{D_\tau} = \bigcup_{\gamma > 0} \mathcal{D}(\gamma,\tau), \hspace{2cm} \mathcal{D} = \bigcup_{\tau \geq 0} \mathcal{D_\tau}. \] 
Recall that for any $\tau > 0$ the set $\mathcal{D}_\tau$ has full Lebesgue measure while $\mathcal{D}_0$ has zero Lebesgue measure. A number in $\mathcal{D}_0$ is said to be of \textit{bounded type}. We now state the main result of this note.

\begin{thm} 
\label{thm: main_thm}
Let  $f$ be an irrational multicritical circle map with unique invariant measure $\mu$. There exists a positive constant $\nu$, depending only on $\rho(f)$, such that the following holds: 
\begin{itemize}
    \item If $\rho(f) \in \mathcal{D}_{\tau} $ for some $\tau \geq 0$ then
    \[ \dim_H(\mu) \geq \dfrac{1}{2\tau +\nu}. \]
    \item If $\rho(f) \notin \mathcal{D}_{\tau}$ for some $\tau > 0$ then 
    \[ \dim_H(\mu) \leq  \dfrac{1}{\tau + 1}.\]
\end{itemize}
\end{thm} 

The upper and lower bounds in Theorem \ref{thm: main_thm} will be proven in Propositions \ref{prop: upper_bound} and \ref{lowerbound} respectively. \\

To end this section let us make a few comments about several related results. Graczyk and \'Swi\k{a}tek studied in \cite{graczyk_singular_1993}  the Hausdorff dimension of the unique invariant measure of multicritical circle maps with rotation number of bounded type. For this particular class the authors show that the associated Hausdorff dimension is bounded away from $0$ and $1$.  Theorem \ref{thm: main_thm} allows to recover the lower bound in \cite{graczyk_singular_1993} while providing an explicit estimate. \\

For circle diffeomorphisms, it follows from the works of Herman \cite{herman_sur_1979} and Yoccoz \cite{yoccoz_conjugaison_1984} that sufficiently regular circle diffeomorphisms are smoothly conjugated to a rigid rotation provided its rotation number is Diophantine. Hence, for any smooth circle diffeomorphism with Diophantine rotation number, its unique invariant measure is equivalent to the Lebesgue measure and therefore its Hausdorff dimension is equal to one.  On the other hand, for any $0 \leq \beta \leq 1$ and any \textit{Liouville} number $\alpha$,  that is, any non-Diophantine irrational number, Sadovskaya \cite{sadovskaya_dimensional_2009} has constructed, using the Anosov-Katok method \cite{anosov_new_1970}, examples of smooth diffeomorphisms with rotation number $\alpha$ whose unique invariant measure has Hausdorff dimension $\beta$.  \\

In the case of \textit{circle homeomorphisms with a break}, i.e. smooth diffeomorphisms with a singular point where the derivative has a jump discontinuity, Khanin and Koci\'c \cite{khanin_hausdorff_2017} have shown that for almost any irrational number $\alpha$  the unique invariant measure of a $C^{2 + \epsilon}$ circle homeomorphism with a break and rotation number equal to $\alpha$ has zero Hausdorff dimension.  \\

\section{Preliminaries}

\subsection{Continued fractions} 

We state some basic results on continued fractions and Diophantine properties. For more details see \cite{lang_introduction_1995}.  Let $\alpha \in (0,1)$ irrational and denote  by $ [a_1,a_2,a_3,\dots]$ its \textit{continued fraction expansion}
\[ \alpha = \cfrac{1}{a_1 + \cfrac{1}{a_2+ \cfrac{1}{a_3 +\cfrac{1}{\cdots}}}} = [a_1,a_2,a_3,\dots].\]
The \textit{n-th convergent} of $\alpha$ is given by \[ \dfrac{p_n}{q_n} = [a_1,a_2,a_3,\dots,a_n] \]
and satisfies 
\[ \min_{1 \leq p, q \leq q_n} \left|  \alpha -\dfrac{p}{q} \right| = (-1)^{n + 1}\left(  \alpha -\dfrac{p_n}{q_n} \right) < \dfrac{1}{q_nq_{n+1}}. \]
If we set $q_{-1} = 0$, $q_0 = 1$ the following recursive relation holds 
\begin{equation}
\label{return_times}
q_{n} = a_nq_{n-1} + q_{n-2},
\end{equation}
for $n \geq 1$. The denominators $q_n$ are called the \textit{return times} of $\alpha$. From (\ref{return_times}) is easy to show that they grow exponentially fast, in fact 
\[ q_n > (\sqrt{2})^n \text{ for } n\geq 2. \]
The Diophantine numbers can be characterized in terms of the continued fraction expansion as follows. Given $\tau \geq 0$ an irrational number $\alpha$ belongs to $ \mathcal{D}_\tau$ if and only if  
\[ \sup_n \dfrac{q_{n+1}}{q_n^{\tau+1}} < \infty. \]
The last inequality is equivalent to 
\[ \sup_n \dfrac{a_{n+1}}{q_n^{\tau}} < \infty. \] 

\subsection{Dynamical Partitions}

\label{sc: dynamical_partitions}
Let $f$ be an orientation preserving circle homeomorphism with irrational rotation number $\alpha = \rho(f) = [a_1, a_2, \dots]$. Using the return times $q_n$ of $\alpha$ described in the previous section we define a partition of $\torus$ as follows.  \\

Fix $x_0 \in \torus$ and let $x_{i} = f^i(x_0)$.  Denote by $I_n$ the circle arc $[x_0, x_{q_n})$ if $n$ is even and $[x_{q_n},x_0)$ if $n$ is odd. Let  $I_n^i = f^i\big(I_n\big)$. Then \[ \mathcal{P}_n(x_0) = \lbrace  I_{n - 1}, I_{n - 1}^1, \dots ,I_{n - 1}^{q_n - 1} \rbrace \cup \lbrace  I_n, I_{n}^1 ,\dots ,I_n^{q_{n-1}-1} \rbrace  \]
defines a partition of $\torus$. For every $x \in \torus$ let $\pn(x)$ be the atom of $\pn$ containing $x$. The next relations follow directly from the continued fraction expansion properties described in the previous section. For all $0 \leq i < q_n$ we have $I_{n + 1}^i\subset I_{n - 1}^i$ and 
\begin{equation}
\label{eq: refining_property}
 I_{n - 1}^i\setminus I_{n + 1}^i= \bigcup_{j = 0}^{a_{n + 1} - 1} I_{n}^{i + q_{n - 1} + jq_{n}}.
 \end{equation}
Hence, it is clear that $\mathcal{P}_{n+1}(x_0)$ is a refinement of $\mathcal{P}_n(x_0)$ where each $I_{n - 1}^i\in \mathcal{P}_{n}(x_0)$ is divided into $a_{n+1} + 1$ pieces of $\pnp(x_0)$. The RHS of (\ref{eq: refining_property}) is the union of $a_{n+1}$ different iterates by $f$ of $I_n$, which are actually adjacent intervals in the circle.  In the following we denote these adjacent intervals by 
 \[ \Delta^{(n)}_{i, j} =  I_{n}^{i + q_{n - 1} + jq_{n}},\]
 for all $0 \leq i < q_n$, $0 \leq j < a_{n + 1}$. To simplify the notation, for $i = 0$, we write simply $\Delta^{(n)}_j$ instead of $\Delta^{(n)}_{0, j}$.  \\
 
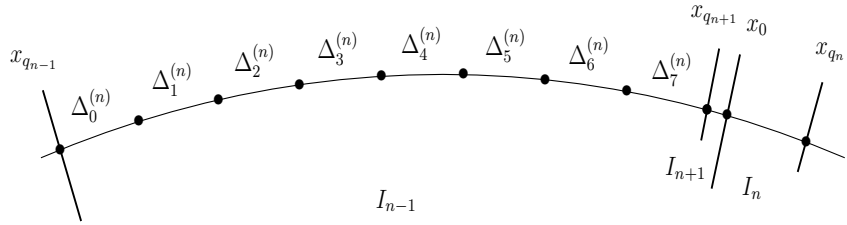
\begin{figure}[h!]
    \centering
    \resizebox{11.5cm}{3cm}
    {
        \begin{tikzpicture}
            \def\xn{72}
            \def\xz{76}
            \def\xnp{77}
            \def\xnm{109}
            \def\r{22}
            \def\uplabel{23.5}
            \def\lowlabel{21}
            \def\lowestlabel{20}
            \draw [domain=70:110] plot ({22 * cos(\x)}, {22 * sin(\x)});
            \foreach \counter  [count=\xi, evaluate=\xi as \xx using int(\xi - 1)] in {1, ..., 8}
            {  
            	\node at (\xnm - \counter*\xz + \counter*\xn + 2:22.5cm) {$\Delta_{\xx}^{(n)}$};	
            	\node at (\xnm - \counter*\xz + \counter*\xn:22cm) {\textbullet};
            }
            \node at (\xnm :\r cm) {\textbullet};
            \node at (\xnp :\r cm) {\textbullet};
            \node at (\xn :\r cm) {\textbullet};
            \node at (\xz :\r cm) {\textbullet};
            \node at (\xn: \uplabel cm) {$x_{q_n}$};
            \node at (\xz - 0.5 : \uplabel cm) {$x_0$};
            \node at (\xnp+ 0.5 : \uplabel cm) {$x_{q_{n + 1}}$};
            \node at (\xnm: \uplabel cm) {$x_{q_{n - 1}}$};
            \node at (\xnm/2 + \xz/2 : \lowestlabel cm) {$I_{n - 1}$};
            \node at (\xnp/2 + \xz/2 + 1: \lowlabel cm) {$I_{n + 1}$};
            \node at (\xz/2 + \xn/2 : \lowlabel cm) {$I_{n}$};
            \foreach \angle [count=\xi, evaluate=\xi as \xx using int(\xi*10)] in {\xn, \xnp}
            	\draw[line width=1pt] (\angle:21.5cm) -- (\angle:23cm);
            \foreach \angle [count=\xi, evaluate=\xi as \xx using int(\xi*10)] in {\xz, \xnm}
            	\draw[line width=1pt] (\angle:20.8cm) -- (\angle:23cm);
        \end{tikzpicture}
	}        
    \caption{Refinement of $I_{n - 1}$ for a rigid rotation with $a_n = 7$.}
\end{figure}
 
Let $f$ be a multicritical circle map with $N$ different critical points. Let $c$ be a critical point of $f$ and consider the dynamical partition associated to it. To simplify the notation let us write $\pn$ instead of $\pn(c)$.  Let
 \[ T_n = f^{q_n}\mid_{\dnm \setminus \dnp}.\] 
Denote
\[ r_n = \# \left\{ 0 <  j < a_{n + 1}\, \Big| \, \overline{\Delta}_j^{(n)} \text{ contains a critical point of } T_n \text{ or } j = a_{n + 1} - 1\right\}.\]
Define
\[ k_0^{(n)} = 0,\hone k_{r_n}^{(n)} = a_{n + 1} - 1,\]
and let 
\[0 < k_1^{(n)} < k_2^{(n)} < \dots < k_{r_n - 1}^{(n)} < a_{n + 1} - 1,\]
 be the indices of the intervals in $\{ \Delta_j^{(n)}\}_{0 <  j < a_{n + 1} - 1}$ whose closure contains a critical point of $T_n$. 
The integers $k_i^{(n)}$ are called the \textit{critical times} of $T_n$ and the corresponding intervals $\Delta_{k_s}^{(n)}$ are called the \textit{critical spots} of $T_n$.  Since the intervals  $I_{n - 1}, I_{n - 1}^1, \dots ,I_{n - 1}^{q_n - 1}$ are disjoint it follows that $T_n$ admits at most $N$ critical points. Therefore 
\[ r_n \leq 2N + 1.\]
 For each $0 \leq s \leq r_n$ we denote by $G_s^{(n)}$ the gap between two consecutive critical spots, namely 
\begin{equation}
\label{eq: bridge}
 G_{s}^{(n)} =  \bigcup_{j = k_s^{(n)}+ 1}^{k_{s + 1}^{(n)}- 1} \Delta^{(n)}_j.
 \end{equation}
 The set $G_s^{(n)}$ is called the $s$-th \textit{bridge} of $I_n$. Notice that each bridge $G_s^{(n)}$ is the union of  $k_{s + 1}^{(n)} - k_s^{(n)} - 1$  adjacent intervals. Following the nomenclature in \cite{estevez_real_2018} we call the bridges $G_0^{(n)}, \dots, G_{r_n}^{(n)}$ \textit{primary bridges} and we refer to its iterates by $f$, which we denote
 \[ G_{i , s}^{(n)} = f^i\left( G_{s}^{(n)}  \right),\]
 as \textit{secondary bridges}. By (\ref{eq: refining_property})

\begin{equation}
\label{eq: critical_spots_bridges}
 I_{n - 1}^i\setminus I_{n + 1}^i= \bigcup_{s = 0}^{r_n} \Delta_{i, k_s^{(n)}}^{(n)} \cup G_{i, s}^{(n)}.
 \end{equation} 
 
In the following we denote the \textit{Lebesgue measure} of a set $A \subset \torus$ by $|A|$. Given a positive constant $M$ and two positive real numbers $x, y$ we write $x \asymp_M y$  if the inequality
\[ M^{-1} y <  x < M y\]
is verified. We now summarise some of the geometric properties of dynamical partitions associated to critical points of multicritical circle maps. All of the properties we state here, except for the first one which goes back to Herman \cite{herman_conjugaison_1988} and \'Swi\k{a}tek \cite{swiatek_rational_1988}, follow from the results of Estevez and de Faria in \cite{estevez_real_2018}. 

\begin{thm}
\label{thm: real_bounds}
There exists a positive constant $M$ such that for any irrational multicritical circle map $f$ of class $C^3$ and for all $n \geq n_0$, where $n_0$ is a sufficiently large natural number depending only on $f$,  the partition $\pn$, associated to any of its critical points, satisfies the following:

\begin{enumerate}
    \item If $I, J$ are any two adjacent atoms of $\pn$ then 
    \[ |I| \asymp_M |J|.\]
    \item For each non-empty bridge $G_{i, s}^{(n)}$ 
    \[  |I_{n - 1}^i| \asymp_M \big|G_{i, s}^{(n)}\big|.\]
    \item For all $0 \leq i < q_{n}$ and   all $0 \leq s \leq r_n$ 
    \[  |I_{n - 1}^i| \asymp_M \left|\Delta_{i, k_s^{(n)}}^{(n)}\right|. \]
    \item For all $0 \leq i < q_{n}$,  all $0 \leq s \leq r_n$ and all $k_s^{(n)}  < j < k_s^{(n + 1)}$\begin{equation}
    \label{eq: quadratic_intervals}
    \big|\Delta^{(n)}_{i, j}\big| \asymp_M \frac{|\dnmi|}{\min\{j - k_s^{(n)}, k_s^{(n + 1)} - j\}^2}.
    \end{equation}
    \item If $\rho(f) = [a_1, a_2, \dots ]$ 
    \begin{equation}
    \label{eq: smallest_interval}
    \min_{I \in \pn} |I| \geq \frac{M^{-n}}{(a_1 a_2 \cdots a_n)^2}.
    \end{equation}
\end{enumerate}
\end{thm}

\begin{figure}[h!]
    \centering
    \resizebox{11.5cm}{3cm}
    {
    	\begin{tikzpicture}
            \draw [domain=70:110] plot ({22 * cos(\x)}, {22 * sin(\x)});
            \def\xn{72}
            \def\xz{82}
            \def\xnp{87}
            \def\xnm{109}
            \def\r{22}
            \def\uplabel{23.5}
            \def\llabel{21.3}
            \def\lowlabel{21}
            \def\lowestlabel{20}
            \def\st{1.5}
            \foreach \counter  [count=\xi, evaluate=\xi as \xx using int(\xi - 1)] in {1, ..., 4}
              	\node[scale=0.02cm] at (\xnm  -  2 * \st  :22cm) {\textbullet};
            \node[scale=0.03cm] at (\xnm  -  2 * \st -  1.5  :22cm) {\textbullet};
            \node[scale=0.03cm] at (\xnm  -  2 * \st -  1.5  - 1.8 / 2.2 :22cm) {\textbullet};
            \node[scale=0.03cm] at (\xnm  -  2 * \st -  1.5  - 1.5  / 2 - 1.8/ 3.2 :22cm) {\textbullet};
            \node[scale=0.03cm] at (\xz/2 + \xnm/2 + 3 +\st + 1.5  :22cm) {\textbullet};
            \node[scale=0.03cm] at (\xz/2 + \xnm/2 + 3 +\st + 1.5 + 1.8 / 2.2  :22cm) {\textbullet};
            \node[scale=0.03cm] at (\xz/2 + \xnm/2 + 3 +\st + 1.5 + 1.5 / 2 + 1.8 / 3.2  :22cm) {\textbullet};
            \node[scale=0.03cm] at (\xz/2 + \xnm/2 + 3 - \st-  2  :22cm) {\textbullet};
            \node[scale=0.03cm] at (\xz/2 + \xnm/2 + 3 - \st -  2  - 1.8 / 2.2 :22cm) {\textbullet};
            \node[scale=0.03cm] at (\xz/2 + \xnm/2 + 3 - \st -  2  - 1.5  / 2 - 1.8/ 3.2 :22cm) {\textbullet};
            \node[scale=0.03cm] at (\xnp + 2*\st  + 2  :22cm) {\textbullet};
            \node[scale=0.03cm] at (\xnp + 2*\st +  2 + 1.8 / 2.2  :22cm) {\textbullet};
            \node[scale=0.03cm] at (\xnp + 2*\st  + 2 + 1.5 / 2 + 1.8 /3.2  :22cm) {\textbullet};
            \node at (\xn: \uplabel cm) {$x_{q_n}$};
            \node at (\xz - 0.5 : \uplabel cm) {$x_0$};
            \node at (\xnp+ 0.5 : \uplabel cm) {$x_{q_{n + 1}}$};
            \node at (\xnm: \uplabel cm) {$x_{q_{n - 1}}$};
            \node [color=red] at (\xz/2 + \xnm/2 + 4: \r cm) {\pgfuseplotmark{square*}};
            \node at (\xnp + \st : \llabel  cm) {$\Delta^{(n)}_{k_2^{(n)}}$};
            \node at (\xz/2 + \xnm/2 + 3: \llabel  cm) {$\Delta^{(n)}_{k_1^{(n)}}$};
            \node at (\xnm - \st: \llabel  cm) {$\Delta^{(n)}_{k_0^{(n)}}$};
            \node at (\xz/2 + \xnm/2 + 3 +\st:\r cm) {\textbullet};
            \node at (\xz/2 + \xnm/2 + 3 - \st:\r cm) {\textbullet};
            \node at (\xnm - 2 * \st:\r cm) {\textbullet};
            \node at (\xnp + 2 * \st:\r cm) {\textbullet};
            \node at (\xnm :\r cm) {\textbullet};
            \node at (\xnp :\r cm) {\textbullet};
            \node at (\xn :\r cm) {\textbullet};
            \node at (\xz :\r cm) {\textbullet};
            \draw[line width=1pt] (\xz/2 + \xnm/2 + 3 +\st:21.7cm) -- (\xz/2 + \xnm/2 + 3 + \st:22.3cm);
            \draw[line width=1pt] (\xz/2 + \xnm/2 + 3 - \st:21.7cm) -- (\xz/2 + \xnm/2 + 3 - \st:22.3cm);
            \draw[line width=1pt] (\xnm - 2 * \st:21.7cm) -- (\xnm -  2 * \st :22.3cm);
            \draw[line width=1pt] (\xnp + 2 * \st:21.7cm) -- (\xnp + 2 * \st :22.3cm);
            \node at (\xnm/2 + \xz/2 : \lowestlabel cm) {$I_{n - 1}$};
            \node at (\xnp/2 + \xz/2: \lowlabel cm) {$I_{n + 1}$};
            \node at (\xz/2 + \xn/2 : \lowlabel cm) {$I_{n}$};
            \foreach \angle [count=\xi, evaluate=\xi as \xx using int(\xi*10)] in {\xn, \xnp}
            	\draw[line width=1pt] (\angle:21.5cm) -- (\angle:23cm);
            \foreach \angle [count=\xi, evaluate=\xi as \xx using int(\xi*10)] in {\xz, \xnm}
            	\draw[line width=1pt] (\angle:20.8cm) -- (\angle:23cm);
    \end{tikzpicture}
    }
    \caption{Refinement of $I_{n - 1}$ for a multicritical circle map with 3 different critical times.}
\end{figure}
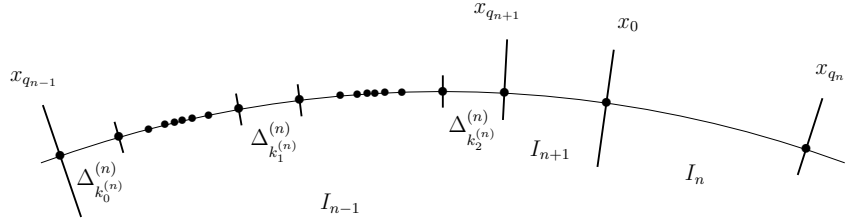

Let us make a few comments about Theorem \ref{thm: real_bounds}. We stress the fact that the constant $M$ in the Theorem is \textit{universal}, that is, independent of the map $f$. Estimates of this type are sometimes called \textit{beau} after the seminal works of Sullivan \cite{sullivan_bounds_1988}. \\

The first assertion is a classical result in the theory of critical circle homeomorphisms and, as we mentioned before, was initially proved by Herman \cite{herman_conjugaison_1988} and \'Swi\k{a}tek \cite{swiatek_rational_1988}. It states the \textit{bounded geometry} of the system, that is, the fact that any two neighbouring atoms of the dynamical partitions have comparable size. This appears in sharp contrast to the case of circle diffeomorphisms where the length ratio of the neighbouring intervals $\dnm, \dn$ in the dynamical partition can be arbitrarily small. In fact, for an irrational rigid rotation $x \mapsto x + \alpha$ with $\alpha = [a_1, a_2, \dots]$, the dynamical partition associated to any point in $\torus$ satisfies 
\[ \frac{\big|\dnm\big|}{\big|\dn\big|} > a_{n + 1}.\]
A recent proof of the first assertion can be found in \cite{estevez_real_2018}. The second and third claims are proven in Proposition 4.1 of \cite{estevez_real_2018}. These assertions can be regarded as the generalization of the following known fact about the geometry of dynamical partitions of multicritical circle maps,  which we quote from \cite{graczyk_singular_1993}: For any element of any dynamical partition, the ratios of its length to the length of the extreme intervals of the next partition subdividing it are bounded by a uniform constant. \\

The last two properties, although not explicitly stated in \cite{estevez_real_2018}, follow directly from the results therein proved. Notice that the last claim is a consequence of the fourth assertion, up to increase the constant $M$ if necessary, by a simple inductive argument. In the unicritical case this was established by de Faria and de Melo in \cite{de_faria_rigidity_1999}. \\

Particular instances of the fourth assertion are used in \cite{estevez_real_2018}, specially in the construction of a \textit{balanced decomposition} for the bridges in the dynamical partition.  For completeness and for the convenience of the reader, let us sketch the proof of this assertion. We start by introducing the main tools appearing in the proof, which, although will not be explicitly used in latter sections, are implicitly at the core of this work. \\

\textbf{Schwarzian derivative.} For a $C^3$ map $f$ defined on an interval or $\torus$ its \textit{Schwarzian derivative} at any non-critical point is given by: 
\[ Sf(x) = \dfrac{D^3f(x)}{Df(x)}-\dfrac{3}{2}\left( \dfrac{D^2f(x)}{Df(x)}\right)^2.\] 

\textbf{Almost parabolic maps.} Let $ a \geq 0$ and $I_0,...,I_{a+1}$ be adjacent intervals on the circle. A negative Schwarzian derivative map 
\[ f: I_0 \cup \dots \cup I_{a}  \rightarrow  I_1 \cup \dots \cup I_{a + 1} \]
such that $f(I_i)=I_{i+1}$ is called an \textit{almost parabolic map}. The intervals $I_0,...,I_{a + 1}$ are called the \textit{fundamental domains} of $f$. The following geometric estimate is due to Yoccoz. See \cite{de_faria_rigidity_1999} appendix B for a proof. 

\begin{lem}
\label{Yoccoz_Lemma}
\textup{\textbf{(Yoccoz's Lemma)}} Let $f: I = I_0 \cup \dots \cup I_a  \rightarrow I_1 \cup \dots \cup I_{a + 1}$ be an almost parabolic map and let $\sigma > 0$ obeying 
\[|I_0|,|I_a|  \geq \sigma |I|.\]
There exists a positive constant $C$, depending only on $\sigma$, such that
\[ \dfrac{C^{-1}|I|}{\min\lbrace j,a-j \rbrace ^2} \leq |I_j| \leq  \dfrac{C|I|}{\min\lbrace j,a-j \rbrace^2} \]
for all $0 <  j < a.$
\end{lem}

\textbf{Distortion estimates.} Given $f : \torus \rightarrow \torus$ and two intervals $M \subset T \subset \torus$ we define the \textit{space of $M$ inside $T$} as
\[ s(M, T) = \min\left\{\frac{|L|}{|M|}, \frac{|R|}{|M|}\right\}\]
where $L$ and $R$ are the left and right components of $T \setminus M$.  A proof of the next Proposition can be found in \cite{de_melo_one-dimensional_1993} (Section IV, Theorem 3.1).

\begin{prop}
\label{Koebe}
\normalfont  \textbf{(Koebe's nonlinearity principle for real maps.)} \textit{Let $f : \torus \rightarrow \torus$ of class $C^3$ and let $\tau, l$ be positive constants. There exists $K(\tau, l, f)$ with the following property: If $T$ is an interval such that $f^i\mid_T$ is a diffeomorphism and if $\sum_{j = 0}^{i - 1} |f^{j}(T)| < l$  then for each interval $M \subset T$ obeying $s(M, T) \geq \tau$ and all $x, y \in M$
\[ \frac{1}{K} \leq \frac{Df^i(x)}{Df^i(y)} \leq K. \]  } 
\end{prop}

We have now all the tools to sketch the proof of the fourth assertion in Theorem \ref{thm: real_bounds}. Proposition 4.2 in \cite{estevez_real_2018} guarantees that for $n$ sufficiently large and for critical times $k_s^{(n)}, k_{s + 1}^{(n)}$ satisfying $$k_{s + 1}^{(n)} - k_s^{(n)} > 3$$  the map 
\[ f^{q_n} \mid_{B_s^{(n)}} :  \bigcup_{j = k_s^{(n)}+ 2}^{k_{s + 1}^{(n)} - 2} \Delta^{(n)}_j \rightarrow  \bigcup_{j = k_s^{(n)}+ 3}^{k_{s + 1}^{(n)} - 1} \Delta^{(n)}_j \]
is almost parabolic. The set 
\[ B_s^{(n)} = \bigcup_{j = k_s^{(n)}+ 2}^{k_{s + 1}^{(n)} - 2} \Delta^{(n)}_j  \subset G^{(n)}_s\]
is called the $s$-th \textit{reduced bridge}. Since  $f^{q_n} \mid_{B_s^{(n)}}$  is almost parabolic, Yoccoz's Lemma together with 1 and 2 of Theorem \ref{thm: real_bounds} imply equation (\ref{eq: quadratic_intervals}) for $i = 0$. The extension to other atoms of the partition is a consequence of Koebe's nonlinearity principle. The fact that the constant in the Theorem can be taken to be universal is a consequence of the beau bounds for cross-ratio inequalities of irrational multicritical maps proven by Estevez and de Faria in \cite{estevez_beau_2018}.

\subsection{Hausdorff dimension}

For a subset $X$ of a metric space $M$ we define its \textit{$d$-dimensional Hausdorff content} by \[ C^d_H(X) := \lim_{\epsilon \rightarrow 0 }\inf_{(U_i)} \sum_i (diam(U_i))^d ,\] where the infimum is taken over all countable covers $(U_i)$ of $X$ satisfying $diam(U_i) < \epsilon$. The \textit{Hausdorff dimension} of $X$ is given by \[ \dim_H(X) := \inf \{ d \geq 0 \mid C^d_H(X) = 0 \}. \] 
We recall that the \textit{Hausdorff dimension of a probability measure} $\mu$ over $M$ is given by
\[\dim_H(\mu) := \inf \{ \dim_H(X) \mid \mu(X) = 1 \}. \] \\ 
A proof of the following result can be found in \cite{przytycki_hausdorff_1989}.

\begin{prop}
\label{prop: frostman_Lemma}
\normalfont \textbf{(Frostman's Lemma)}. \textit{Suppose that $\mu$ is a probability Borel measure on the interval and that for $\mu$-a.e. point \[ \delta_1 \leq  \liminf_{\epsilon \rightarrow 0 } \dfrac{\log \mu(x-\epsilon, x+\epsilon)}{\log \epsilon} \leq \delta_2.\]
Then \[ \delta_1 \leq \dim_H(\mu) \leq \delta_2. \]}
\end{prop}

\section{Proof of Theorem \ref{thm: main_thm}}

In this section $f$ will denote a $C^3$ irrational multicritical circle map with $N$ different critical points, rotation number $\alpha$ and unique invariant measure $\mu$.  We fix one of its critical points and denote by $\{ \pn\}_{n \in \N}$ its associated dynamical partitions. We will use the notations introduced in the preliminaries to denote the continued fraction and return times of $\alpha$.

 \subsection{Upper bound}
 
 We will bound from above the Hausdorff dimension of $\mu$ by constructing appropriate full $\mu$-measure sets whose Hausdorff dimension we can control. Before giving their explicit definition let us first try to motivate the construction. \\
 
 Notice that for $a_{n + 1}$ sufficiently large the union of the  ``big'' atoms of the partition $\pn$, namely the  union of the intervals $I_{n -1}, I_{n-1}^i, \dots, I_{n - 1}^{q_n - 1},$  has $\mu$-measure close to $1$. Indeed, by definition of the dynamical partitions
\begin{equation} 
\label{eq: partition_measure}
1 = q_{n + 1} \delta_{n} + q_{n}\delta_{n + 1},
\end{equation}
 for all $n \geq 1$, where $$\delta_n = \mu(I_n).$$  By  (\ref{eq: partition_measure})
\begin{equation}
\label{eq: upper_bound_delta_n}
 \delta_n \leq \frac{1}{q_{n + 1}}.
 \end{equation}
 From (\ref{return_times}) and (\ref{eq: upper_bound_delta_n})
\begin{equation}
\label{eq: small_intervals_measure} 
q_{n-1} \delta_n  \leq \dfrac{1}{a_{n+1} a_n}.
\end{equation}
Hence, by (\ref{eq: partition_measure}) and (\ref{eq: small_intervals_measure})
\[ \mu\left( \bigcup_{i = 0}^{q_n - 1} I_{n - 1}^j \right) = q_n \delta_{n - 1}  \geq 1 - \frac{1}{a_na_{n + 1}}.\]
Moreover
\[ \mu\left( \bigcup_{i = 0}^{q_n - 1} I_{n - 1}^j  \setminus I_{n + 1}^j \right) = q_n( \delta_{n - 1} - \delta_{n + 1}) \geq  1 - \frac{2}{a_{n + 1}}.\]
Let us recall equation (\ref{eq: critical_spots_bridges}), namely 
\[ I_{n - 1}^i\setminus I_{n + 1}^i =  \bigcup_{s = 0}^{r_n} \Delta^{(n)}_{i, k_s^{(n)}} \cup  G^{(n)}_{i, s},\]
where $k_0^{(n)} , k_1^{(n)} , \dots, k_{r_n}^{(n)} $ are the critical times associated to $f^{q_n}\mid_{I_{n -1} \setminus I_{n + 1}}$ and $ \Delta^{(n)}_{i, k_s^{(n)}}$, $G^{(n)}_{i, s}$ are, respectively, the iterates of the critical spots and primary bridges defined in section \ref{sc: dynamical_partitions}.  Recall also that the secondary bridge $G^{(n)}_{i, s}$ is the union of $k_{s + 1}^{(n)} - k_s^{(n)} - 1$ adjacent intervals 
\[ G^{(n)}_{i, s} = \bigcup_{j = k_s^{(n)} + 1}^{k_{s + 1} - 1} \Delta^{(n)}_{i, j}.\]
By  (\ref{eq: quadratic_intervals}) in Theorem \ref{thm: real_bounds} and provided $k_s^{(n + 1)} - k_s^{(n)}$ is sufficiently big,  the ``central intervals''  of  the secondary bridge $G_{i, s}^{(n)}$ will have a much smaller Lebesgue measure when compared to the first and last intervals. Nevertheless, as they are all iterates of $I_n$, their measure with respect to $\mu$ must be the same.  Having this in mind we construct an appropriate full measure set  to bound the Hausdorff dimension of the invariant measure $\mu$. \\

Given $0 < \gamma < 1$, denote by $G^{(n)}_{s,\gamma}$ the union of the $k^{(n)}_{s + 1} - k^{(n)}_{s} - 2 \lfloor a_{n + 1}^\gamma \rfloor$ central intervals of $G_{s}^{(n)}$ when seen as the union of adjacent atoms in $\mathcal{P}_{n + 1}$, namely 
\[G^{(n)}_{s,\gamma} = \bigcup \left\{  \Delta^{(n)}_j \, \Big| \, k_s^{(n)} + a_{n + 1}^{\gamma}< j < k_s^{(n + 1)} + 1 - a_{n + 1}^{\gamma}\right\}.\]
Let
\begin{equation}
\label{eq: central_intervals}
A^n_{i,\gamma} = \bigcup_{s = 0}^{r_n} f^i\left( G_{s, \gamma}^{(n)} \right), \hone 
 A^n_\gamma = \bigcup_{i = 0}^{q_n - 1} A^n_{i,\gamma},
 \end{equation}
for all $0 \leq i < q_n$.  We have the following.

\begin{lem}
\label{full_measure_cover}
Let $M, n_0$ as in Theorem \ref{thm: real_bounds} when applied to $f$. For any $0 <\gamma < 1$ and all $n \geq n_0$ the following holds:
\begin{enumerate}
    \item $\mu(A^n_{i,\gamma}) \geq \delta_n (a_{n + 1} - (4N + 2)a_{n + 1}^\gamma).$
    \item $\mu(A_\gamma^n) \geq \left( 1 - \frac{2}{a_{n + 1}} \right)\left( 1-\frac{4N+2}{a_{n+1}^{1 - \gamma}} \right).$
    \item $|A^n_{i,\gamma}| \leq  \frac{(4N + 2)M|\dnmi|}{\lfloor a_{n+1}^\gamma \rfloor}.$ 
    \item $|A_\gamma^n| \leq \frac{(4N + 2)M}{\lfloor a_{n+1}^\gamma \rfloor}.$
\end{enumerate}
\end{lem}

\begin{proof}
By definition of $A_{i, \gamma}^n$ and the invariance of $\mu$ by $f$ we have
\begin{align*}
\mu(A_{i, \gamma}^n) & \geq \delta_n \sum_{s = 0}^{r_n} \big(k_{s + 1}^{(n)} - k_s^{(n)} -  2a_{n + 1}^\gamma \big)  \\
& \geq \delta_n (a_{n + 1} - (4N + 2)a_{n + 1}^\gamma) 
\end{align*}
which proves the first assertion. By (\ref{eq: partition_measure}) and (\ref{eq: small_intervals_measure})
 \begin{align*}
  \mu(A_\gamma^n) &= q_n\mu(A_{0, \gamma}^n) \\
  & \geq   q_n \delta_n a_{n + 1}\left(1 - \frac{4N + 2}{a_{n + 1}^{1 -\gamma}}\right) \\
& = (1-q_n\delta_{n+1} - q_{n-1}\delta_n)\left(1 - \frac{4N + 2}{a_{n + 1}^{1 -\gamma}}\right) \\
& \geq \left( 1 - \frac{2}{a_{n + 1}} \right)\left(1 - \frac{4N + 2}{a_{n + 1}^{1 -\gamma}}\right).
\end{align*}
This proves the second claim. By (\ref{eq: quadratic_intervals})
\begin{align*}
|A_{i, \gamma}^n| & = \sum_{s \in K_\gamma^{(n)}} \sum_{j = k_s^{(n)} + \lfloor a_n^{\gamma} \rfloor + 1}^{k_{s + 1}^{(n)} - \lceil a_{n + 1}^\gamma \rceil} |\Delta_j^{(n)}| \\
& \leq 2M  \sum_{s = 0}^{r_n} |\dnmi|\sum_{j = \lfloor a_n^{\gamma} \rfloor + 1}^{+\infty}  \frac{1}{j^2} \\
& = \frac{(4N + 2)M|\dnmi|}{\lfloor a_{n + 1}^\gamma \rfloor}
\end{align*}
which proves the third assertion. Since the intervals $I_{n-1}^0, I_{n - 1}^1, \dots, I_{n - 1}^{q_n - 1}$ are disjoint the last assertion follows directly from the third one.
\end{proof}
\begin{cor}
\label{prop: singular_measure}
Suppose $\rho(f)$ is not of bounded type. Then the unique invariant measure of $f$ is singular with respect to the Lebesgue measure.
\end{cor}

\begin{proof}
This follows directly from $3$ and $4$ of the previous Lemma.
\end{proof}

\begin{prop}
\label{prop: upper_bound}
Suppose $\rho(f) \notin \mathcal{D}_\tau$ for some $ \tau >  0$. Then  
\[ \dim_H(\mu) \leq  \dfrac{1}{\tau + 1}.\]
\end{prop} 

\begin{proof}
Since $\rho(f) \notin \mathcal{D}_\tau$ there exists an increasing  sequence $(n_k)_{k \in \N}$ of natural numbers obeying 
\[ a_{n_k+1} \geq q_{n_k}^{\tau}\]
for all $k \in \N$. Given $0 < \gamma < 1$, let $A_{i, \gamma}^n$, $A_\gamma^n$ as in (\ref{eq: central_intervals}) and define
\[A_\gamma = \bigcap_{i \geq 1}\bigcup_{k\geq i} A^{n_k}_\gamma.\]
By Lemma  \ref{full_measure_cover} 
\[\mu(A_\gamma^{n_k}) \xrightarrow[k \to \infty]{} 1.\] 
Thus
 \[\mu(A_\gamma) = 1.\]
 By definition of Hausdorff dimension 
\[ \dim_H(\mu)  \leq \inf_{0 < \gamma < 1}\dim_H(A_\gamma).\]
Let us show that $\dim_H(A_\gamma) \leq \frac{1}{\tau + 1}$ for $\gamma$ sufficiently close to $1$. Let $d > \frac{1}{\tau+1}$, $\epsilon > 0$ and take $K$ sufficiently large so that $\text{diam}(\mathcal{P}_{n_K}) < \epsilon$. Thus 
\[ \mathcal{C} = \left\{ A_{i, \gamma}^{n_k} \mid K < k, \, 0 \leq i < q_{n_k} \right\} \]
   is an open cover of $A_\gamma$ with diameter less than $\epsilon$. Hence 
\begin{align*} 
    C^d_H(A_\gamma) & \leq \liminf_{K \rightarrow \infty} \sum_{ k > K} \sum_{i=0}^{q_{n_k}-1} |A_{i, \gamma}^{n_k}|^d \\
    & \leq \liminf_{K \rightarrow \infty} \sum_{ k > K}  \dfrac{M^d(4N + 2)^d}{\lfloor a_{n_k +1}^{\gamma} \rfloor^d } \sum_{i=0}^{q_{n_k}-1} |I^i_{n_k-1}|^d \\
    & \leq \liminf_{K \rightarrow \infty} \sum_{ k > K}  \dfrac{M^d(4N + 2)^d}{a_{n_k +1}^{d\gamma}} q_{n_k}^{1-d} \\
    &   \leq \liminf_{K \rightarrow \infty} \sum_{ k > K}  M^d(4N + 2)^d q_{n_k}^{1-d(\gamma\tau + 1)}.
\end{align*}
By hypothesis $1 < d(\tau+1)$. Since the return times $q_n$ grow at least exponentially the sum in the last inequality converges for $\gamma$ sufficiently close to 1. Thus
\[ C^d_H(A_\gamma) = 0\]
for $\gamma$ sufficiently close to $1$. Therefore
\[ \inf_{0 < \gamma < 1} \dim_H (A _\gamma) \leq \frac{1}{\tau + 1} \]
which finishes the proof. 
\end{proof}

\subsection{Lower bound}

The lower bound for the Hausdorff dimension of the unique invariant measure of $f$ will be a direct application of Frostman's Lemma (Proposition \ref{prop: frostman_Lemma}). 

\begin{prop}
\label{lowerbound}
 Suppose $\rho(f) \in \mathcal{D}_\tau$ for some $\tau \geq 0$.  Let $M > 1$ be as in Theorem \ref{thm: real_bounds}. Then 
\[ \dim_H(\mu) \geq \dfrac{1}{2\tau +\nu_1 + \nu_2\log M },\]
where
 \[\nu_1 = \limsup_{n \to \infty} \frac{2\log(a_1a_2\dots a_{n})}{\log q_n}, \hone \nu_2 = \limsup_{n \to \infty} \frac{n}{\log q_n}.\]
\end{prop}

\begin{proof}
By Theorem \ref{thm: real_bounds} there exists a  natural number $n_0$ such that
\[ \min_{\Delta \in \mathcal{P}_{n + 1}} |\Delta| > \frac{M^{-(n + 1)}}{(a_1 \cdots a_{n + 1})^2},\] 
for all $n \geq n_0$. Let $x \in \torus$ such that $x$ is not an end point of any of the intervals of the dynamical partitions $\{ \pn\}_{n \in \N}$. Let
\[ \Gamma = \sup_{n \in \N} \frac{a_{n + 1}}{q_n^{\tau}}.\]
Notice that $\Gamma < +\infty$ since $\rho(f) \in \mathcal{D}_\tau$.   Let  $0 < \epsilon < 1$ and define 
\[ n(x, \epsilon) = \min \left\{k \in \N \,\mid\, \exists \Delta \in \mathcal{P}_{k + 1} \text{ s.t. } \Delta \subset B_\epsilon(x) \right\}.\]
Let $\epsilon$ sufficiently small so that $n = n(x, \epsilon) > n_0$. Notice that $B_\epsilon(x)$ must be contained in the union of two adjacent elements  $\Delta_1,\Delta_2 \in \mathcal{P}_n$. Let $\Delta \in \pnp$ such that $\Delta \subset B_\epsilon(x)$.  Hence
\[ \epsilon \geq |\Delta| , \hone \mu(B_\epsilon(x)) \leq 2 \mu(\Delta_0^{(n-1)}) \leq 2 q_n^{-1},\]
which yields to
\begin{align*}
\dfrac{\log \mu(B_\epsilon(x))}{\log \epsilon} & \geq \dfrac{\log 2\mu(\Delta_0^{(n-1)})}{\log |\Delta|} \\
& \geq \dfrac{\log q_n - \log 2}{\log (a_1a_2\dots a_{n+1})^2 + \log M^{n+1}} \\
& \geq \dfrac{\log q_n - \log 2}{2\tau\log (\Gamma q_n) + 2\log(a_1a_2...a_n) + (n+1)\log M }.
\end{align*}
Therefore
\[ \liminf _{\epsilon \to 0}  \dfrac{\log \mu(B_\epsilon(x))}{\log \epsilon}  \geq \frac{1}{2\tau + \nu_1 + \nu_2 \log M}.\]
The result follows by Frostman's Lemma.
\end{proof}

\bibliographystyle{acm}
\bibliography{bibliography.bib}

\end{document}